\documentclass{conm-p-l}
\usepackage{amssymb} 
\usepackage{amsmath}
\usepackage{liemacs-red}

\newcommand{\cM}{\mathcal M}
\newcommand{\cK}{\mathcal K}

\newcommand{\bO}{\mathbb O} 
\newcommand{\str}{{\mathfrak{str}}}

\newcommand{\AU}{\mathop{{\rm AU}}\nolimits}

\newcommand{\Bij}{\mathop{{\rm Bij}}\nolimits}
\newcommand{\Mod}{\mathop{{\rm Mod}}\nolimits}

\newcommand{\Cau}{\mathop{{\rm Cau}}\nolimits}
\newcommand{\Conf}{\mathop{{\rm Conf}}\nolimits}
\newcommand{\Conj}{\mathop{{\rm Conj}}\nolimits}
\newcommand{\Inv}{\mathop{{\rm Inv}}\nolimits}

\renewcommand{\phi}{\varphi}

\newcommand{\Stand}{\mathop{{\rm Stand}}\nolimits}

\renewcommand\mlabel{\label}

\newtheorem{theorem}{Theorem}[section]
\newtheorem{lemma}[theorem]{Lemma}
\newtheorem{proposition}[theorem]{Proposition} 
\newtheorem{cor}[theorem]{Corollary} 

\theoremstyle{definition}
\newtheorem{definition}[theorem]{Definition}
\newtheorem{example}[theorem]{Example}

\theoremstyle{remark}
\newtheorem{remark}[theorem]{Remark}
\newtheorem{prob}[theorem]{Problem}

\numberwithin{equation}{section}

\begin{document}

\title{On the geometry of standard subspaces}


\author[Neeb]{Karl-Hermann Neeb}
\address{Department Mathematik, Universit\"at Erlangen-N\"urnberg, 
Cauerstrasse 11, 91058 Erlangen, Germany}
\email{neeb@math.fau.de}

\thanks{K.-H.~Neeb acknowledges supported by DFG-grant NE 413/9-1.}


\date{July 9, 2017}

\begin{abstract}
A closed real subspace $V$ of a complex Hilbert space $\cH$ 
is called standard if $V \cap i V = \{0\}$ and $V + i V$ is dense in $\cH$. 
In this note we study several aspects of the geometry of the space 
$\Stand(\cH)$ of standard subspaces. In particular, we show that 
modular conjugations define the structure of a reflection space and that 
the modular automorphism groups extend this to the structure of a dilation space. 
Every antiunitary representation of a graded 
Lie group $G$ leads to a morphism of dilation spaces 
$\Hom_{\rm gr}(\R^\times,G) \to \Stand(\cH)$. Here dilation invariant 
geodesics (with respect to the reflection space structure) 
correspond to antiunitary representations $U$ of $\Aff(\R)$ 
and they are decreasing if and only if $U$ is a positive energy representation. 
We also show that the ordered symmetric spaces corresponding to 
euclidean Jordan algebras have natural order 
embeddings into 
$\Stand(\cH)$ obtained from any antiunitary positive energy 
representations of the conformal group. 
\end{abstract}

\maketitle
\tableofcontents

\section*{Introduction} 

A closed real subspace $V$ of a complex Hilbert space $\cH$ 
is called {\it standard} 
if $V \cap i V = \{0\}$ and $V + i V$ is dense in $\cH$ 
(\cite{Lo08}). We write $\Stand(\cH)$ for the set of standard subspaces of~$\cH$.
The main goal of this note is to shed some light on the geometric structure 
of this space and how it can be related to 
geometric structures on manifolds on which Lie groups $G$ act via 
antiunitary representations on~$\cH$. 

Standard subspaces arise naturally in the modular theory of 
von Neumann algebras. If $\cM \subeq B(\cH)$ is a von Neumann algebra 
and $\xi \in \cH$ is a cyclic separating vector for $\cM$, 
i.e., $\cM\xi$ is dense in $\cH$ and the map 
$\cM \to \cH, M\mapsto M\xi$ is injective, then 
\[ V_\cM := \oline{ \{ M\xi \: M^* = M, M \in \cM \}} \] 
is a standard subspace of $\cH$. Conversely, 
one can associate to every standard subspace $V \subeq \cH$ 
in a natural way a von Neumann algebra in 
the bosonic and fermionic Fock space of $\cH$, and this 
assignment has many nice properties 
(see \cite[\S\S 4,6]{NO17} and \cite{Lo08} for details). This establishes a direct 
connection between standard subspaces and pairs 
$(\cM,\xi)$ of von Neumann algebras with cyclic separating vectors. 
Since the latter objects play a key role in 
Algebraic Quantum Field Theory in the context of Haag--Kastler nets 
(\cite{Ar99, Ha96, BW92}), it is important 
to understand the geometric structure of the space~$\Stand(\cH)$. 
Here a key point is that it reflects many important properties 
of von Neumann algebras related to modular inclusions and 
symmetry groups quite faithfully in a much simpler environment 
(\cite[\S 4.2]{NO17}). 
We refer to \cite{Lo08} for an excellent survey on this correspondence. 
In QFT, standard subspaces provide 
the basis for the technique of modular localization, developed 
by Brunetti, Guido and Longo in \cite{BGL02}. 

Every standard subspace $V$ determines by the polar decomposition 
of the closed operator $S$, defined on $V + i V$ by 
$S(x + i y) = x- iy$, a pair $(\Delta_V, J_V)$ of so-called modular objects, 
i.e., $\Delta_V$ is a positive selfadjoint operator and 
$J_V$ is a conjugation (an antiunitary involution) satisfying 
$J_V \Delta_V J_V = \Delta_V^{-1}$. 
This correspondence leads to a bijection 
\[ \Psi \:  \Mod(\cH) \to \Stand(\cH), \quad 
(\Delta, J) \mapsto  \Fix(J\Delta^{1/2}) \] 
between the set $\Mod(\cH)$ of pairs of modular objects 
and $\Stand(\cH)$. 

There actually is a third model of $\Stand(\cH)$ that 
comes from the fact that each pair $(\Delta, J)$ defines a homomorphism 
\[ \gamma \: \R^\times \to \AU(\cH)\quad \mbox{ by } \quad 
\gamma(e^t) := \Delta^{-it/2\pi}, \quad 
\gamma(-1) := J, \] 
where $\AU(\cH)$ denotes the group of unitary and antiunitary operators on $\cH$.
This perspective will play a crucial role in our analysis of $\Stand(\cH)$. 

In Section~\ref{sec:1} we discuss Loos' concept of a reflection 
space, which is a generalization of the 
concept of a symmetric space. Although symmetric spaces play a central 
role in differential geometry and harmonic analysis 
for more than a century, reflection spaces never received much 
attention. As we shall see below, they provide exactly the right 
framework to study the geometry of $\Stand(\cH)$. 
Reflection spaces are specified in terms of a system 
$(s_x)_{x \in M}$ of involutions satisfying 
\[ s_x(x) = x \quad \mbox{ and } \quad s_x s_y s_x = s_{s_xy}
\quad \mbox{ for } \quad x,y \in M.\] 
One sometimes has even more 
structure encoded in a family $(r_x)_{r \in \R^\times, x \in M}$ 
of $\R^\times$-actions on $M$ satisfying 
\[ r_x(x) = x, \quad r_x s_x = (rs)_x \quad \mbox{ and } \quad 
r_x s_y r_x^{-1} = s_{r_xy}\quad \mbox{  for } \quad x,y \in M, \ r,s \in \R^\times.\]
 This defines the structure 
of a {\it dilation space}, a concept studied in the more general 
context of $\Sigma$-spaces by Loos in \cite{Lo72}. 
For $r = -1$, we obtain a reflection space, so that 
a dilation space is a reflection space with additional structure. 
Other important classes of dilation spaces 
are the ruled spaces discussed in \cite[Ch.~VI]{Be00} 
that arise naturally in Jordan theory. 

In Section~\ref{sec:2} we turn to the space $\Stand(\cH)$ 
of standard subspaces and show that 
it carries a natural dilation space structure. 
This corresponds naturally to dilation space structures on the sets 
$\Mod(\cH)$ and $\Hom(\R^\times, \AU(\cH))$. 
The underlying reflection space structure on $\Stand(\cH)$ is given by 
\[ V_1 \bullet V_2 := s_{V_1} V_2 = J_1 J_2 V_2\] 
and the map 
\[ q \: \Stand(\cH) \to \Conj(\cH), \quad V \mapsto J_V \]  
onto the symmetric space $\Conj(\cH)$ of antiunitary involutions on $\cH$ 
is a morphism of reflection spaces. Here an interesting point is 
that $\Conj(\cH)$ does not carry a non-trivial dilation space structure, 
so that the weaker notion of a reflection space on $\Stand(\cH)$ 
actually leads to the much richer dilation space structure. 
If $(G,\eps_G)$ is a graded topological group, i.e., 
$\eps_G \to \{\pm1\}$ is a continuous homomorphism, 
then, for every antiunitary representation $U \:  G  \to \AU(\cH)$, 
the natural map $U_* \: \Hom_{\rm gr}(\R^\times,G) \to 
\Hom_{\rm gr}(\R^\times, \AU(\cH))$ defines a morphism of 
dilation spaces $\cV_U \: \Hom_{\rm gr}(\R^\times,G) \to \Stand(\cH)$ 
which is the Brunetti--Guido--Longo (BGL) 
map $\cV_U$ from \cite[Prop.~5.6]{NO17} and \cite[Thm.~2.5]{BGL02}. 

A morphism of reflection spaces $\gamma \: \R \to \Stand(\cH)$ 
is called a geodesic. In Proposition~\ref{prop:2.9} 
we describe the geodesics for which $q \circ \gamma$ is continuous 
in terms of unitary one-parameter groups $(U_t)_{t \in \R}$. They are of the form 
\[ \gamma(t) = U_t V, \quad \mbox{ where } \quad 
J_V U_t J_V = U_{-t} \quad \mbox{ for } \quad t \in \R.\] 
On the other hand, we have for each $V \in \Stand(\cH)$ 
the corresponding dilation group implemented 
by the unitary operators $(\Delta_V^{it})_{t \in \R}$. 
Both structures interact nicely for geodesics invariant under the 
dilation group. In Proposition~\ref{prop:2.12} we show that, if $(U_t)_{t \in \R}$ 
does not commute with the dilations $(\Delta_V^{is})_{s \in \R}$, 
the geodesic is an orbit of 
$\Aff(\R)_0$ in $\Stand(\cH)$, where the action is given by an antiunitary 
representation. 

A particularly intriguing structure on $\Stand(\cH)$ 
is the order structure defined by set inclusion to which we turn 
in Section~\ref{sec:3}. This structure is trivial if $\cH$ is finite dimensional 
and it is also trivial on the subspace 
\[ \Stand_0(\cH) := \{ V \in \Stand(\cH) \: V + i V = \cH\}.\] 
But if $\cH$ is infinite dimensional non-trivial 
inclusions can be obtained from antiunitary positive energy 
representations of $\Aff(\R)$, which actually lead to monotone 
dilation invariant geodesics (Theorem~\ref{thm:3.1}). 
This is a direct consequence of the 
Theorems of Borchers and Wiesbrock (cf.~\cite{Lo08}, \cite{NO17}) 
and the dilation space structure thus provides a new geometric perspective 
on these results that were originally formulated in terms of inclusions 
of von Neumann algebras (\cite{Bo92}, \cite{Wi93}). 
In view of this characterization of the monotone dilation invariant geodesics,   
it is an interesting open 
problem to characterize all monotone geodesics in $\Stand(\cH)$. 
To get some more information on the ordered space $\Stand(\cH)$, one natural 
strategy is to consider finite dimensional submanifolds, resp., 
orbits $\cO_V := U_{G_1}.V\cong G_1/G_{1,V}$, where $U$ is an antiunitary representation. 
Then 
\begin{equation}
  \label{eq:semigroup}
S_V := \{ g \in G_1 \: U_gV \subeq V \} 
\end{equation} 
is a closed subsemigroup of $G_1$ with $G_{1,V} = S_V \cap S_V^{-1}$ 
and $S_V$ determines an order structure on $G_1/G_{1,V}$ by 
$g G_{1,V} \leq g' G_{1,V}$ if $g \in g'S_V$ for which the inclusion 
$G_1/G_{1,V} \into \Stand(\cH)$ is an equivariant order embedding. Of course, 
the most natural cases arise if 
$V$ corresponds to some $\gamma \in \Hom_{\rm gr}(\R^\times,G)$ 
under the BGL construction and then $G_{1,\gamma} \subeq G_{1,V}$, so that 
$\cO_V$ is a $G_1$-equivariant quotient of $G_1/G_{1,\gamma} \cong G_1.\gamma \subeq 
\Hom_{\rm gr}(\R^\times,G)$. We conclude this note by showing that, 
if $G$ is the conformal group of a euclidean Jordan algebra $E$ 
and $\gamma \:\R^\times \to G$ corresponds to scalar multiplication 
on $E$, the ordered homogeneous spaces $U_{G_1}V$, 
$V := \cV_U(\gamma)$, obtained from an antiunitary positive energy 
representations $(U,\cH)$ of~$G$, are mutually isomorphic 
and the order structure can be described 
by showing that the semigroup $S_V$ coincides with the well-known 
Olshanski semigroup $S_{E_+}$ of conformal compressions of the open positive 
cone $E_+$ (\cite{HN93}, \cite{Ko95}). This result is based on the 
maximality of the subsemigroup $S_{E_+}$ in $G_1$ which is 
proved in an appendix. 

{\bf Acknowledgment:} We are most grateful to Wolfgang Bertram 
for illuminating discussions on the subject matter of this note 
and for pointing out several crucial references, 
such as \cite{Lo67}. We also thank Jan M\"ollers for suggestions 
to improve earlier drafts of the manuscript. 

\section{Reflection spaces}
\mlabel{sec:1}

In this first section we first review some generalities on reflection 
spaces (\cite{Lo67}) and introduce the notion of a dilation space 
by specialization of Loos' more general concept of a 
$\Sigma$-space (\cite{Lo72}). A key feature of these abstract 
concepts is that they work well in many categories, in particular 
in the category of sets and the category of topological spaces and 
not only in the category of smooth manifolds. 
Only when it comes the finer geometric points related to the 
concept of a symmetric space, a smooth structure is required. 
In Section~\ref{sec:2} this will be crucial for the space 
$\Stand(\cH)$ which carries no natural smooth structure 
but which is fibered over the topological space $\Conj(\cH)$, 
endowed with the strong operator topology.

\begin{definition} (a) Let $M$ be a set and 
\[ \mu \: M \times M \to M, \quad (x,y) \mapsto x \bullet y =: s_x(y) \] 
be a map with the following properties: 
\begin{itemize}
\item[\rm(S1)] $x \bullet x=x$ for all $x \in M$, i.e., $s_x(x) = x$.
\item[\rm(S2)] $x \bullet (x \bullet y) =y$ for all $x,y \in M$, i.e., $s_x^2 = \id_M$. 
\item[\rm(S3)] $s_x(y \bullet z) = s_x(y)\bullet s_x(z)$ for all $x,y \in M$, 
i.e., $s_x \in \Aut(M,\bullet)$. 
\end{itemize}
Then we call $(M,\mu)$ a {\it reflection space} (\cite{Lo67, Lo67b}). 

(b) If $M$ be a smooth manifold and $\mu \: M \times M \to M$ is a smooth map 
turning $(M,\mu)$ into a reflection space, then it is called a 
smooth reflection space. If, in addition, 
each $x$ is an isolated fixed point of $s_x$, 
then it is called a {\it symmetric space} (in the sense of Loos). 

If $M$ is a topological space and $\mu$ is continuous, we call it a 
{\it topological reflection space}. 

(c) If $(M,\bullet)$ and $(N,\bullet)$ are reflection spaces, then a map 
$f \: M \to N$ is called a {\it morphism of reflection spaces} if 
\[ f(m \bullet m') = f(m) \bullet f(m') \quad \mbox{ for } \quad 
m,m' \in M.\]
\end{definition}



\begin{example} \mlabel{ex:1.3}
(a) Any group $G$ is a reflection space 
with respect to the product
\begin{equation}
  \label{eq:e1}
g \bullet h := s_g(h) := gh^{-1}g.
\end{equation}

Note that left and right translations 
\[ \lambda_g(x) = gx \quad \mbox{ and }\quad 
\rho_g(x) = xg \] 
are automorphisms of the reflection space $(G,\bullet)$. 

The subset $\Inv(G)$ of involutions in $G$ is a reflection subspace 
on which the product takes the form $s_g(h) := ghg = ghg^{-1}$. 

(b) Suppose that $G$ is a group and $\tau \in \Aut(G)$ is an involution. 
For any subgroup $H\subeq G^\tau := \Fix(\tau)$, we obtain on the coset space 
$M := G/H$ the structure of a reflection space by 
\begin{equation}
  \label{eq:e2}
xH \bullet y H := x\tau(x)^{-1} \tau(y) H.
\end{equation}
For this reflection space structure all left translations 
$\mu_g \: G/H \to G/H, xH \mapsto gxH$ are automorphisms. 

If, in addition, $G$ is a Banach--Lie group and $H$ is a complemented 
Lie subgroup, then $G/H$ is a smooth reflection space. It is a symmetric space 
if and only if $H$ is an open subgroup of the group $G^\tau$. 
In fact, if $e_M := eH$ is the base point of $G/H$, then 
$T_{e_M}(G/H) \cong \g/\fh$ and the tangent map of $s_{e_M}(gH) = \tau(g)H$ 
is the involution on $\g/\fh$ induced by $T_e(\tau)$. This equals 
$-\id_{\g/\fh}$ if and only if $\fh = \g^\tau$, which is equivalent to 
$H$ being an open subgroup of $G^\tau$. 
For any open subgroup $H\subeq G^\tau = \Fix(\tau)$, we thus obtain 
by \eqref{eq:e2} on $G/H$ the structure of a symmetric space. 

Note that $H = \{e\}$ is also allowed, showing that 
\begin{equation}
  \label{eq:e4}
x \bullet y  := x\tau(x)^{-1} \tau(y)
\end{equation}
also defines on $G$ the structure of a smooth reflection space. 

(c)  Every manifold $M$ is a smooth reflection space with respect to 
$x \bullet y := y$ for $x,y \in M$. 

(d) If $(M,\bullet)$ and $(N, \bullet)$ are smooth reflection spaces, then 
so is their product $M \times N$ with respect to 
\[ (m,n) \bullet (m',n') := (m \bullet m', n \bullet n').\] 

(e) If $(M,\bullet)$ is a reflection space and $q \: M \to N$ is a surjective submersion 
whose kernel relation is a congruence relation with respect to 
$\bullet$, i.e., $q(x) = q(x')$ and $q(y) = q(y')$ implies 
$q(x \bullet y) = q(x' \bullet y')$, then 
\[ q(x) \bullet q(y) := q(x \bullet y) \] 
defines on $N$ the structure of a reflection space. 

In fact, that the product on $N$ is well-defined is our assumption. 
That it is smooth follows from the smoothness of the map 
$M \times M \to M/N, (x,y) \mapsto q(x \bullet y)$ and the fact that 
$q \times q \:  M \times M \to N \times N$ is a submersion. 
Now the relations (S1-3) for $N$ follow immediately from the corresponding 
relations on $M$. 

(f) In addition to (b), we consider a smooth action $\alpha \: H \to \Diff(F)$ 
of $H$ on the manifold $F$ and consider the space 
\[ M := G \times_H F = (G \times F)/H, \] 
where $H$ acts on $G \times F$ by $h.(g,f) := (gh^{-1}, \alpha_h(f))$. 
We write $[g,f]$ for the $H$-orbit of $(g,f)$ in $M$. 
We claim that 
\begin{equation}
  \label{eq:mult1}
[g_1,f_1] \bullet [g_2, f_2] := [g_1 \tau(g_1)^{-1} \tau(g_2), f_2] 
\end{equation}
defines on $M$ the structure of a smooth reflection space on which 
$G$ acts by automorphisms via 
$\mu_g[g',f] = [gg', f]$. 

That \eqref{eq:mult1} is a well-defined smooth binary operation is clear. 
That we obtain a reflection space is most naturally derived from (b), (c), (d) and (e). 
First we note that the product manifold $G \times F$ carries a natural reflection space 
structure given by 
\[ (gH, f) \bullet (g'H, f') := (g\tau(g)^{-1} \tau(g'), f'), \] 
where we use the reflection space structure from (a) on $G$, and the trivial 
one from (c) on~$F$. Next we note that the quotient map 
$q(g,f) := [g,f]$ is a submersion and that, for 
$g,g' \in G, h,h' \in H$, and $f,f' \in F$, the image 
\[ q((g,f)\bullet (g',f')) = [g\tau(g)^{-1}\tau(g'), f'] \] 
of the product in $G/H \times F$ does not change on the $H$-orbits: 
\[ q((gh^{-1},\alpha_h(f))\bullet (g'h',\alpha_{h'}(f'))) 
= [g\tau(g)^{-1}\tau(g')(h')^{-1}, \alpha_{h'}f'] 
= [g\tau(g)^{-1}\tau(g'), f'].\] 
Therefore our claim follows from (e). 

One of the main results of \cite{Lo67, Lo67b} asserts that every 
finite dimensional connected reflection space $(M,\bullet)$ 
is of this form, where 
\begin{itemize}
\item[$\bullet$] $G := \la s_x s_y \: x,y \in M\ra_{\rm grp} \subeq \Diff(M)$ 
carries a finite dimensional Lie group structure. 
\item[$\bullet$] $H = G^\tau$ for $\tau(g) = s_e g s_e$, where $e \in M$ is a 
base point. 
\item[$\bullet$] $F := \{ m \in M \: s_m = s_{e}\}$. 
\end{itemize}

Typical examples with discrete spaces $F$ arise for 
$F := \pi_0(H)$ on which $H$ acts through the quotient homomorphism 
$H \to \pi_0(H)$ by translations. 

We also note that $G$ acts transitively on $G\times_H F$ if and only if 
$H$ acts transitively on $F$. For any $f \in F$ we then have 
$G \times_H F \cong G/H_f$ as a homogeneous space of $G$. 

(g) If $(V,\beta)$ is a $\K$-vector space ($\ch(\K) \not =2$) 
and $\beta \: V \times V \to \K$ is a symmetric bilinear form, then the 
subset $V^\times := \{ v \in V \: \beta(v,v) \not=0\}$ is a reflection 
space with respect to 
\[ x \bullet y = s_x(y) := -y + 2 \frac{\beta(x,y)}{\beta(x,x)}x.\] 
Note that $s_{\lambda x} = s_x$ for every $\lambda \in \K^\times$ 
and, conversely, that $s_x = s_z$ implies $z \in \K^\times x$ because 
$z \in \ker(s_x + \1) = \K x$. 

For $\K = \R$ and $V$ a locally convex space, we thus obtain 
on $V^\times$ the structure of a real smooth reflection space and each 
level set 
\[ V^\times_m = \{ x \in V \: \beta(x,x) = m \} \] 
becomes a symmetric space. The same holds for the image 
$V^\times/\K^\times$ in the projective space~$\bP(V)$. 
\end{example}

\begin{remark} (S1-3) can also be formulated as conditions on the map \break 
$s_x \: M \to M$, namely that $s_x$ is an involution fixing $x$, and (S3) takes the 
form 
\[ s_{x \bullet y} = s_x s_y s_x = s_x \bullet s_y,\] 
where $\bullet$ on the right hand side refers to the natural reflection space 
structure on the set $\Inv(\Bij(M))$ of involutions in the group $\Bij(M)$ of 
permutations of $M$ (Example~\ref{ex:1.3}(a)). 
\end{remark}

\subsection{Powers and geodesics in reflection spaces}

\begin{definition} \mlabel{def:quad} (Quadratic representation and powers) 
Let $(M,\bullet)$ be a reflection space and $e \in M$ be a base point. 
The map 
\[ P = P_ e \:  M \to \Bij(M), \quad 
P_e(m) := s_m s_e \] 
is called the {\it quadratic representation of $M$ with respect to $e$}. 
For $x \in M$, we define the {\it powers with respect to $e$} by 
$x^0 := e$, $x^1 := x$ and 
\[  x^{n+2} := P(x)x^n = x\bullet (e \bullet x)
\quad \mbox { for } \quad n \geq 0, 
\quad \mbox{ and } \quad 
x^n := x^{-n} \quad \mbox{ for } \quad n \in -\N.\] 
\end{definition} 
An easy induction then shows that 
  \begin{equation}
    \label{eq:pow1}
x^n \bullet x^m = x^{2n-m} \quad \mbox{ for } \quad n,m \in \Z.
\end{equation}
We also note that, if $f \: (M,\bullet) \to (M',\bullet)$ is a morphism 
of reflection spaces and $e' = f(e)$, then 
\begin{equation}
  \label{eq:pow2}
f(x^n) = f(x)^n \quad \mbox{ for } \quad x \in M, n \in\Z.  
\end{equation}
Note that \eqref{eq:pow1} means that the map 
$(\Z, \bullet) \to (M,\bullet), n \mapsto x^n$ 
is a morphism of reflection spaces if $\Z$ carries the canonical reflection 
space structure (Example~\ref{ex:1.3}(a)).

\begin{definition} If $(M,\bullet)$ is a (topological) reflection space, then we 
call a (continuous) morphism $\gamma \: (\R, \bullet) \to (M,\bullet)$ of 
reflection spaces a {\it geodesic}. 
\end{definition}

\begin{theorem}\mlabel{thm:oeh} 
{\rm(Oeh's Theorem, \cite{Oe17})} 
Let $G$ be a topological group. Then the geodesics 
$\gamma \: \R \to G$ with $\gamma(0) = g$ are the curves of the form 
\[  \gamma(t) = \eta(t)g, \] 
where $\eta \in \Hom(\R,G)$ is a continuous one-parameter group. 
The range of $\gamma$ is contained in $\Inv(G)$ if and only if 
\[ g\in \Inv(G) \quad \mbox{ and } \quad 
g \eta(t) g^{-1} = \eta(-t) \quad \mbox{ for }\quad t \in\R.\] 
\end{theorem}

\begin{proof} Since right multiplication with $g^{-1}$ is an automorphism 
of the reflection space 
$(G,\bullet)$, we may w.l.o.g.\ assume that $g = e$ and show that 
in this case the geodesics are the continuous one-parameter groups. 

Clearly, every one-parameter group $\gamma \: \R \to G$ is also a morphism
of reflection spaces, hence a geodesic. 
Suppose, conversely, that $\gamma$ is a geodesic with $\gamma(0) = e$. 
From \eqref{eq:pow2} and the relation $t^n = nt$ 
in the pointed reflection space $(\R,\bullet,0)$, 
it follows that 
\[ \gamma(nt) = \gamma(t)^n \quad \mbox{ for }\quad t \in  \R, n \in \Z.\] 
It follows in particular, that the restriction of $\gamma$ to any 
cyclic subgroup $\Z t \subeq \R$ is a group homomorphism. 
Applying this to $t = \frac{1}{n}$, $n \in \N$, we see that 
$\gamma\res_{\Q} \: \Q \to G$ is a group homomorphism. 
Now the continuity of $\gamma$ implies that $\gamma$ is a homomorphism. 

The second assertion is trivial. 
\end{proof}

\subsection{Dilation spaces} 

Although the reflection space structure is a key bridge between 
manifold geometry and transformation groups 
(\cite{Lo67}), there are natural situations where one has additional 
structures encoded by a family of actions $\mu_x \: \Sigma \to \Diff(M)$ 
of a given Lie group $\Sigma$ on $M$ such that 
$x$ is fixed under $\mu_x$  and the family $(\mu_x)_{x \in M}$ 
satisfies a certain compatibility condition similar to (S3). 
This leads to 
the notion of a $\Sigma$-space introduced by O.~Loos in (\cite{Lo72}). 
Here we shall need only the special case $\Sigma = \R^\times$, 
so that we shall speak of dilation spaces. Restricting to the 
subgroup $\{\pm 1\} \subeq \R^\times$, we then obtain a reflection space, 
so that dilation spaces are reflection spaces with additional 
point symmetries encoded in $\R^\times$-actions parametrized by the points of~$M$. 

\begin{definition} Let $M$ be a set and suppose we are 
given a map 
\[ \mu \: M \times \R^\times \times M \to M, \quad (x,r,y) 
\mapsto x \bullet_r y  := r_x(y) := \mu_r(x,y) := \mu(x,r,y)  \] 
with the following properties: 
\begin{itemize}
\item[\rm(D1)] $r_x(x) = x$ for every $x \in M$ and $r \in \R^\times$. 
\item[\rm(D2)] $r_x \circ s_x = (rs)_x$ for $x \in M$, $r,s \in \R^\times$.
\item[\rm(D3)] $r_x(y \bullet_s z) = r_x(y)\bullet_s r_x(z)$ for $x,y \in M$, 
$r,s \in \R^\times$, i.e., $r_x \in \Aut(M,\mu)$. 
\end{itemize}
Then we call $(M,\mu)$ a {\it dilation space}. 
\end{definition}

\begin{remark} (a) If $(M,\mu)$ is a dilation space, then 
$(M,\mu_{-1})$ is a reflection space. 

(b) In \cite[Def.~VI.2.1]{Be00} the notion of a 
{\it ruled space} is defined as a smooth 
dilation space $(M,\mu)$ with the additional property that, for 
every $r \in \R^\times$, the tangent map 
$T_x(r_x)$ is diagonalizable with eigenvalues $r$ and $r^{-1}$. 
This ensures that $(M,\mu_{-1})$ is a symmetric space. 
We refer to \cite[Thm.~VI.2.2]{Be00} for a characterization of the ruled 
spaces among symmetric spaces. 
\end{remark}

\begin{example} \mlabel{ex:1.13} 
(a) For every group $G$, the space $M := \Hom(\R^\times, G)$ is a 
dilation space with 
\begin{equation}
  \label{eq:r-dil}
(\gamma \bullet_r \eta)(s) := \gamma(r) \eta(s) \gamma(r)^{-1}
\quad \mbox{ for } \quad \gamma,\eta \in M, r,s \in \R^\times.
\end{equation}
Here (D1/2) are clear. For (D3) we calculate 
\begin{align*}
\big(\gamma \bullet_r (\eta \bullet_s \zeta)\big)(t) 
&= \gamma(r) \eta(s) \zeta(t) \eta(s)^{-1}\gamma(r)^{-1}\\
&= \big(\gamma(r) \eta(s) \gamma(r)^{-1}\big)\big(
\gamma(r)\zeta(t)  \gamma(r)^{-1}\big)\big(\gamma(r)
\eta(s)^{-1}\gamma(r)^{-1}\big)\\
&= [(\gamma\bullet_r \eta) \bullet_s (\gamma \bullet_r \zeta)](t).
\end{align*}

For $r =-1$, \eqref{eq:r-dil} specializes to the reflection space structure given by 
\[ (\gamma \bullet \eta)(s) := \gamma(-1) \eta(s) \gamma(-1)^{-1}.\]

(b) For every $\gamma \in \Hom(\R^\times, G)$, the 
conjugacy orbit $G.\gamma = \{ \gamma^g \: g \in G\}$ with 
$\gamma^g(t) := g\gamma(t)g^{-1}$ is a dilation subspace. 
This follows directly from 
\begin{align*}
(\gamma^{g_1} \bullet_r \gamma^{g_2})(s) 
&= g_1 \gamma(r)g_1^{-1} g_2 \gamma(s) g_2^{-1}  g_1 \gamma(r)^{-1}g_1^{-1} \\
&= (g_1 \gamma(r)g_1^{-1} g_2) \gamma(s) 
(g_1 \gamma(r)g_1^{-1} g_2)^{-1}. 
\end{align*}
Note that $G.\gamma = G_1.\gamma$ follows from the fact that 
$\gamma(-1) \in G_\gamma$. 

(c) If $G$ is a Lie group with Lie algebra $\g$, we represent any smooth 
$\gamma \in \Hom(\R^\times, G)$ 
by the pair $(\gamma'(0),\gamma(-1)) \in \g \times \Inv(G)$. 
The corresponding set of pairs is 
\[ \cG := \{ (x,\sigma) \in \g \times \Inv(G) \: \Ad_\sigma x  = x\}\] 
and the reflection structure takes on this set the form 
\[ (x,\sigma) \bullet (y,\eta) = (\Ad_\sigma y, \sigma \eta\sigma) 
= (\Ad_\sigma y, \sigma\bullet \eta).\]
The additional dilation space structure from (a) is given by 
\[ (x,\sigma) \bullet_{e^t} (y,\eta) = (e^{t \ad x} y, \exp(tx)\eta \exp(-tx)).\] 

For $(x,\sigma) \in \cG$, the $G$-orbit is the set 
\begin{equation}
  \label{eq:ad-orb1}
\cO_{(x,\sigma)} := \{ (\Ad_g x, g\sigma g^{-1}) \: g \in G \} 
\end{equation}
which has an obvious fiber bundle structure 
$G \times_{G^\sigma} \Ad_{G^\sigma} x \to G/G^\sigma$ over the 
symmetric space $G/G^\sigma$. 

In many interesting cases (see f.i.~\cite[\S 4.3]{NO17}), 
we actually have $G_x \subeq G^\sigma$. 
This is in particular the case if $\Ad_\sigma = e^{\pi i\ad x}$ and 
$G = G_0 \{e,\sigma\}$ (\cite{NO17}). Then 
\begin{equation}
  \label{eq:ad-orb2}
\cO_{(x,\sigma)} \cong  \Ad_G x 
\end{equation}
simply is an adjoint orbit in $\g$. 
\end{example}

On homogeneous spaces, dilation space structure can sometimes 
be constructed along the lines of Example~\ref{ex:1.3}(b) 
if one considers subgroups with a central one-parameter group: 

\begin{proposition} \mlabel{prop:1.11}
Let $G$ be a group and 
$\alpha \:  \R^\times \to \Aut(G)$ be a homomorphism. Then 
\[ g \bullet_r h := g \alpha_r(g^{-1}h) \] 
defines on $G$ the structure of a dilation space for which 
left translations $\lambda_g(x) = gx$ act by automorphisms. 
If $H \subeq G$ is a subgroup fixed pointwise by $\alpha$, then 
the space $G/H$ of left $H$-cosets is a dilation space with respect to 
\[ g H \bullet_r uH := g \alpha_r(g^{-1}u)H \quad \mbox{ for } \quad 
g,u \in G. \] 
\end{proposition}

\begin{proof}
Here (D1) is trivial and (D2) follows from $r_g = \lambda_g  \alpha_r \lambda_g^{-1}$ 
for $g \in G$. For (D3) we observe that 
\[ g \bullet_r (h \bullet_s u) 
= g \alpha_r\big(g^{-1}h \alpha_s(h^{-1}u)\big) 
= g \alpha_r\big(g^{-1}h) \alpha_{rs}(h^{-1}u)\] 
equals 
\begin{align*}
(g \bullet_r h) \bullet_s (g \bullet_r u) 
&= g \alpha_r(g^{-1}h)  \alpha_s\big(
 \alpha_r(g^{-1}h)^{-1}g^{-1} g \alpha_r(g^{-1}u)\big) \\
&= g \alpha_r(g^{-1}h)  \alpha_s\big( \alpha_r(h^{-1}u)\big)
= g \alpha_r(g^{-1}h)  \alpha_{rs}(h^{-1}u).  
\end{align*}

The second assertion follows immediately from the first one 
because the binary operations $\bullet_r$ are obviously well-defined 
on $G/H$ and (D1-3) for $G/H$ follows immediately from the 
corresponding relations for $G$. 
\end{proof}

By specialization we immediately obtain: 
\begin{example} \mlabel{ex:1.15} 
(a) Every real affine space $A$ is a dilation space with respect to 
\[ a \bullet_r b = a + r(b-a) = (1-r)a + r b 
\quad \mbox{ for }  \quad a,b \in A, r \in \R^\times.\] 

(b) Let $V$ be a vector space and 
$\alpha \: \R^\times \to \GL(V)$ be a group homomorphism. Then 
\[ a \bullet_r b = a + \alpha_r(b-a) \] 
defines on $V$ the structure of a dilation space. 
\end{example}

If $(M,\mu)$ is a dilation space and $e \in M$ is a base point, 
then $\alpha_e(r)(x) := e \bullet_r x$ defines a homomorphism 
$\alpha_e \: \R^\times \to \Aut(M,\bullet_{-1})_e$ 
whose range is central. Conversely, we have: 

\begin{cor} Let $(M,\bullet)$ be a reflection space
and $G \subeq \Aut(M,\bullet)$ be a subgroup 
acting transitively on~$M$. We fix a base point $e \in M$ 
and a homomorphism $\beta \: \R^\times_+ \to G_e$ 
with central range. Then $\bullet_{-1} := \bullet$ together 
with 
\begin{equation}
  \label{eq:hom-dil}
g.e \bullet_r y 
:= g \beta(r)g^{-1}.y \quad \mbox{ for } \quad g \in G, r \in \R^\times_+, y \in M 
\end{equation}
defines on the reflection space $(M,\bullet)$ 
the structure of a dilation space on which $G$ acts by automorphisms. 
\end{cor}

\begin{proof} Let $H := G_e$, so that we may identify 
$M$ with $G/H$. Since the elements $\beta(r)$, $r \in \R^\times$, 
fix $e$, they commute with the reflection $s_e$. 
Therefore $\alpha_r(g) := \beta(r)g\beta(r)^{-1}$ for 
$r \in \R^\times_+$ and $\alpha_{-1}(g) := s_e g s_e$ 
define a homomorphism $\alpha \: \R^\times \to \Aut(G)$ 
and $H = G_e$ is fixed pointwise by each~$\alpha_r$.  
We now obtain with Proposition~\ref{prop:1.11} on $G/H$ 
the structure of a dilation space by 
\[ g H \bullet_r uH := g \alpha_r(g^{-1}u)H \quad \mbox{ for } \quad 
g,u \in G. \] 
For $r = -1$, this means that 
\[ g H \bullet_{-1} uH 
= g s_{eH} g^{-1}u s_{eH} H 
= g s_{eH} g^{-1}u H 
= g(e H \bullet g^{-1}uH) 
= g H \bullet uH \] 
recovers the given reflection space structure. 
For $r > 0$ we find 
\[ g H \bullet_r uH 
= g \beta(r) g^{-1}u \beta(r)^{-1}H 
= g \beta(r) g^{-1}u H = (g \beta(r)g^{-1}).uH,\] 
and this coincides with \eqref{eq:hom-dil}. 
\end{proof}

\subsection{Geodesics in dilation spaces} 

Since dilation spaces $(M,\mu)$ are reflection spaces with 
additional structure, geodesics in $(M,\bullet)$ are not always 
compatible with the dilation structure. The following lemma characterizes those 
which are. 

\begin{lemma} If $(M,\mu)$ is a 
dilation space and $\lambda \in \R^\times$, 
then the following are equivalent for a geodesic 
$\gamma \: \R \to (M,\bullet)$: 
\begin{itemize}
\item[\rm(a)] $\gamma$ is a morphism of dilation spaces, 
where the dilation structure on $(\R,\bullet)$ is given by 
$t \bullet_r s := (1-r^\lambda)t + r^\lambda s = t + r^\lambda(s-t)$ for $r > 0$. 
\item[\rm(b)] $\gamma((1-r^\lambda)t + r^\lambda s) 
= \gamma(t) \bullet_r \gamma(s)$ for 
$t,s \in \R, r \in \R^\times_+$. 
\item[\rm(c)] $\gamma(r^\lambda s) = \gamma(0) \bullet_r \gamma(s)$ for 
$s \in \R, r \in \R^\times$. 
\end{itemize}
\end{lemma}

\begin{proof} The equivalence of (a) and (b) is by definition. 
Further, (c) follows from (b) by specializing to $t = 0$. 
If (c) is satisfied, then we obtain
\begin{align*}
\gamma(t \bullet_r s) 
&
 = \gamma(t + r^\lambda(s-t)) = \gamma\big({\textstyle\frac{t}{2}} \bullet r^\lambda(t-s)\big) \\ 
&= \gamma\big({\textstyle\frac{t}{2}}\big)  \bullet \gamma(r^\lambda(t-s)) 
{\buildrel{(c)}\over{=}} \gamma\big({\textstyle\frac{t}{2}}\big) 
\bullet_{-1} \big(\gamma(0) \bullet_r \gamma(t-s)\big) \\
&= \big(\gamma({\textstyle\frac{t}{2}}) \bullet_{-1} \gamma(0)\big) 
\bullet_r \big(\gamma({\textstyle\frac{t}{2}}) \bullet_{-1} \gamma(t-s)\big) \\
&= \gamma\big({\textstyle\frac{t}{2}} \bullet 0\big) 
\bullet_r \gamma\big({\textstyle\frac{t}{2}} \bullet (t-s)\big) 
= \gamma(t) \bullet_r \gamma(t - (t-s)) 
= \gamma(t) \bullet_r \gamma(s).  
\qedhere\end{align*}
\end{proof}

\begin{remark} The main difference between Examples~\ref{ex:1.15}(a) and (b) 
is that, for $0 \not=v \in V$, the geodesic 
$\gamma(t) = t v$ is a morphism of dilation spaces for the 
$\lambda$-dilation structure on $\R$ if and only if 
\[ r^\lambda tv = \gamma(r^\lambda t) =  0 \bullet_r tv  = \alpha(r)tv 
\quad \mbox{ for } \quad 
r \in \R^\times_+, t \in \R.\] 
This is equivalent to $\alpha(r)v = r^\lambda v$ for all $r \in \R^\times_+$. 
Therefore the geodesics which are morphisms of dilation spaces 
are generated by the elements of the common eigenspace 
\[ V_\lambda := \bigcap_{r \in \R^\times_+} \ker(\alpha(r) - r^\lambda \1).\] 
\end{remark}

\section{The space of standard subspaces} 
\mlabel{sec:2}

We now apply the general discussion of reflection and dilation spaces 
to the space $\Stand(\cH)$ of standard subspaces and its relatives, 
the space $\Mod(\cH)$ of pairs of modular objects $(\Delta, J)$ 
and the space $\Hom_{\rm gr}(\R^\times, \AU(\cH))$ of continuous 
antiunitary representations of $\R^\times$. 

\begin{definition} A closed real subspace $V \subeq \cH$ is called a {\it standard 
subspace} if $V \cap i V = \{0\}$ and $V + i V$ is dense in $\cH$. 
We write $\Stand(\cH)$ for the set of standard subspaces of~$\cH$.
\end{definition}

For every standard subspace $V \subeq \cH$, we obtain an 
antilinear unbounded operator 
\[ S \: \cD(S) := V + i V \to \cH, \qquad 
S(v + i w) := v - i w \] 
with $V = \Fix(S) = \ker(S- \1)$. 
The operator $S$ is closed, so that $\Delta_V := S^*S$ is a positive 
selfadjoint operator. We thus obtain 
the polar decomposition 
\[ S = J_V \Delta_V^{1/2},\] 
where $J_V$ is an antilinear involution and 
the modular relation $J_V\Delta_V J_V = \Delta_V^{-1}$ is satisfied 
(cf.\ \cite[Prop.~3.3]{Lo08}, \cite{NO16}). 

We write $\Mod(\cH)$ for the set of pairs $(\Delta, J)$, where 
$J$ is a {\it conjugation} (an antilinear bijective isometry) and 
$\Delta$ is a positive selfadjoint operator with $J\Delta J = \Delta^{-1}$. 
Then the map 
\begin{equation}
  \label{eq:bij1}
\Phi \: \Mod(\cH) \to \Stand(\cH), \quad 
\Phi(\Delta, J) = \Fix(J \Delta^{1/2}) 
\end{equation}
is a bijection (\cite[Prop.~3.2]{Lo08}). 

To see more geometric structure on $\Stand(\cH)$, we have to 
connect its elements to homomorphisms $\R^\times \to \AU(\cH)$. 
This is best done in the context of graded groups and their 
antiunitary representations. 

\begin{definition} \mlabel{def:grad-grp} 
(a) A {\it graded group} is a pair $(G,\eps_G)$ consisting 
of a group $G$ and a surjective homomorphism $\eps_G \: G \to \{\pm 1\}$. 
We write $G_1 := \ker \eps_G$ and $G_{-1} = G \setminus G_1$, so that 
\[ G = G_1 \dot\cup G_{-1} \quad \mbox{ and } \quad 
G_j G_k = G_{jk} \quad \mbox{ for } \quad j,k \in \{\pm 1\}.\] 
Often graded groups 
are specified as pairs $(G,G_1)$, where $G_1$ is a subgroup of index~$2$, 
so that we obtain a grading by $\eps_G(g) := 1$ for $g \in G_1$ and 
$\eps_G(g) :=-1$ for $g \in G \setminus G_1$.

If $G$ is a Lie group and $\eps_G$ is continuous, i.e., $G_1$ is an open 
subgroup, then $(G,\eps_G)$ is called a {\it graded Lie group}. 

If $G$ is a topological group with two connected components, then 
we obtain a canonical grading for which $G_1$ is the identity component. 
Concrete examples are $\R^\times, \GL_n(\R), \OO_n(\R)$ and the group 
$\AU(\cH)$ of unitary and antiunitary operators on a complex 
Hilbert space $\cH$, endowed with the strong operator topology. 

(b) A {\it morphism of graded groups} $\phi \: (G,\eps_G) \to (H,\eps_H)$ 
is a group homomorphism $\phi \: G \to H$ with $\eps_H \circ \phi = \eps_G$, 
i.e., $\phi(G_j) \subeq H_j$ for $j =1,-1$. We write $\Hom_{\rm gr}(G,H)$ for the set of 
(continuous) graded homomorphism between the (topological) 
graded groups $(G,\eps_G)$ and $(H,\eps_H)$. 

(c) For a complex Hilbert space $\cH$, the group $\AU(\cH)$ carries a natural 
grading defined by $\eps(U) = -1$ if $U$ is antiunitary and 
$\eps(U) = 1$ if $U$ is unitary. 
For a topological graded group $(G,\eps_G)$, an 
{\it antiunitary representation} $(U,\cH)$ 
is a continuous homomorphism $U \: G \to \AU(\cH)$ 
of graded groups, where $\AU(\cH)$ carries the strong operator topology.
\end{definition}

It is easy to see that 
\[ \Psi \: \Mod(\cH) \to \Hom_{\rm gr}(\R^\times, \AU(\cH)), \quad 
\Psi(\Delta, J)(e^t) := \Delta^{-it/2\pi}, \quad 
\Psi(\Delta, J)(-1) := J\] 
defines a bijection (\cite[Lemma~2.22]{NO17}). Combining this with $\Phi$ from 
\eqref{eq:bij1}, we obtain a bijection 
\begin{equation}
  \label{eq:calV}
 \cV := \Phi \circ \Psi^{-1} \: 
\Hom_{\rm gr}(\R^\times, \AU(\cH)) \to \Stand(\cH).
\end{equation}

\begin{theorem} We obtain the structure of a dilation space 
  \begin{itemize}
  \item on $\Hom_{\rm gr}(\R^\times, \AU(\cH))$ by 
\[ (\gamma \bullet_r \eta)(t) := \gamma(r) \eta(t) \gamma(r)^{-1}, \] 
  \item on $\Mod(\cH)$ by 
\[ (\Delta_1, J_1) \bullet_r 
 (\Delta_2, J_2) 
=\begin{cases} 
(J_1 \Delta_2^{-1} J_1, J_1 \bullet J_2) & \text{ for } r = -1 \\
(\Delta_1^{-it/2\pi} \Delta_2 \Delta_1^{it/2\pi} ,
\Delta_1^{-it/2\pi} J_2 \Delta_1^{it/2\pi})
& \text{ for } r = e^t,    
\end{cases}\]
\item and on $\Stand(\cH)$ by 
\[ V_1 \bullet_r V_2 =
\begin{cases}
  J_{V_1} J_{V_2} V_2 & \text{ for } \quad r = -1 \\ 
\Delta_{V_1}^{-it/2\pi} V_2 & \text{ for } \quad r = e^t.
\end{cases}\]
  \end{itemize}
The bijections 
$\Phi \: \Mod(\cH) \to \Stand(\cH)$ and 
$\Psi \: \Mod(\cH) \to \Hom_{\rm gr}(\R^\times, \AU(\cH))$ 
are isomorphisms of dilation spaces. 
\end{theorem}

\begin{proof} First we observe that, for each $r \in \R^\times$ 
we obtain by $\bullet_r$ a  binary operations on the spaces 
$\Hom_{\rm gr}(\R^\times, \AU(\cH))$, 
$\Mod(\cH)$ and $\Stand(\cH)$, respectively. 
In particular, $\Hom_{\rm gr}(\R^\times, \AU(\cH))$ is a dilation subspace 
of $\Hom(\R^\times, \AU(\cH))$ (Example~\ref{ex:1.13}), 
hence a dilation space. 
It therefore remains to show that 
$\Phi$ and $\Psi$ are compatible with all binary 
operations $\bullet_r$ and this implies in particular that (D1-3) are satisfied 
on $\Mod(\cH)$ and $\Stand(\cH)$. 

First we consider $\Psi$. 
Let $\gamma_j = \Psi(\Delta_j, J_j)$ for $j =1,2$, $r \in \R^\times$, 
and \break $\gamma := \Psi((\Delta_1, J_1) \bullet_r (\Delta_2, J_2))$. 
For $r = -1$ we then have 
\[ \gamma(-1) = J_1 J_2 J_1 = \gamma_1(-1) \gamma_2(-1) \gamma_1(-1) 
= (\gamma_1 \bullet_{-1} \gamma_2)(-1), \] 
and, for $t \in \R$, we have 
\[ \gamma(e^t) 
= (J_1 \Delta_2^{-1} J_1)^{-it/2\pi}
= J_1 \Delta_2^{-it/2\pi} J_1 
= \gamma_1(-1) \gamma_2(e^t) \gamma_1(-1)
= (\gamma_1 \bullet \gamma_2)(e^t).\] 
For $r = e^t$, $t\in \R$, the pair 
$(\Delta_1, J_1) \bullet_r (\Delta_2, J_2)$ is obtained from 
$(\Delta_2, J_2)$ by conjugating with $\Delta_1^{-it/2\pi} = \gamma_1(e^t)$ 
and this immediately implies that 
$\gamma = \gamma_1 \bullet_r \gamma_2$. 
We conclude that $\Psi$ is an isomorphism of dilation spaces. 

Now we turn to $\Phi$. From 
\begin{eqnarray} 
  \label{eq:stand-ref}
 \Fix((J_1 J_2 J_1) (J_1 \Delta_2^{-1} J_1)^{1/2})
&=& \Fix(J_1 J_2 \Delta_2^{-1/2} J_1)\\
&=& J_1 \Fix(J_2 \Delta_2^{-1/2}) = J_1 V_2' = J_1 J_2 V_2 \notag 
\end{eqnarray} 
it follows that $\Phi$ is compatible with $\bullet_{-1}$. 
For $r = e^t$, $t \in \R$,  the pair \break 
$(\Delta_1, J_1) \bullet_r (\Delta_2, J_2)$ is obtained from 
$(\Delta_2, J_2)$ by conjugating with $\Delta_1^{-it/2\pi}$, so that 
\[ \Phi\big((\Delta_1, J_1) \bullet_r (\Delta_2, J_2)\big) 
= \Delta_1^{-it/2\pi} V_2. \] 
This completes the proof.
\end{proof}

Any antiunitary representation of a graded group $G$ leads to a 
map \break $\cV_U \: \Hom_{\rm gr}(\R^\times,G) \to \Stand(\cH)$, 
an observation due to Brunetti, Guido and Longo (\cite[Thm.~2.5]{BGL02}; 
see also \cite[Prop.~5.6]{NO17}). 
The naturality of this map immediately shows that it is a morphism 
of dilation spaces: 

\begin{cor} {\rm(The BGL (Brunetti--Guido--Longo) construction)} 
\mlabel{cor:bgl} 
Let $(U,\cH)$ be an antiunitary representation of 
the graded topological group $(G, \eps_G)$. Then 
\[ \cV_U := \cV \circ U_* \: \Hom_{\rm gr}(\R^\times, G) \to \Stand(\cH), \quad 
\gamma \mapsto  \cV(U \circ \gamma) \] 
is a morphism of dilation spaces. 
\end{cor}

\begin{proof} Since $\Phi$ and $\Psi$ are isomorphisms of dilation spaces, 
it suffices to observe that 
\[ U_* \: \Hom_{\rm gr}(\R^\times, G) \to \Hom_{\rm gr}(\R^\times, \AU(\cH)), 
\quad \gamma \mapsto U \circ \gamma \] 
is a morphism of dilation spaces. But this is a trivial 
consequence of the fact that $U$ is a morphism of graded topological groups. 
\end{proof}

In addition to the dilation structure, the space 
$\Stand(\cH)$ carries a natural involution $\theta$: 

\begin{remark} (The canonical involution) \\
(a) On $\Hom_{\rm gr}(\R,\AU(\cH))$ the involution 
$\theta(\gamma) =  \gamma^\vee$, 
$\gamma^\vee(t) := \gamma(t^{-1})$, defines an isomorphism 
of reflection spaces which is compatible with the dilation space structure in the 
sense that 
\[ (\gamma \bullet_r \eta)^\vee = \gamma \bullet_r (\eta^\vee)
\quad \mbox{ for }\quad \gamma,\eta \in 
\Hom_{\rm gr}(\R,\AU(\cH))\] 
because 
\[ (\gamma \bullet_r \eta)^\vee(s) 
= \gamma(r) \eta(s^{-1}) \gamma(r)^{-1} 
= (\gamma \bullet_r \eta^\vee)(s).\] 
The corresponding automorphism is given on 
$\Stand(\cH)$ by $\theta(V) := V'$ and 
on $\Mod(\cH)$ by $\theta(\Delta, J) = (\Delta^{-1},J)$. 

The fixed points of $\theta$ correspond to 
\begin{itemize}
\item graded homomorphisms $\gamma \: \R^\times \to \AU(\cH)$ with 
$\R^\times_+ \subeq \ker\gamma$. 
\item elements $V \in \Stand(\cH)$ which are 
Lagrangian subspaces of the symplectic vector space $(\cH,\omega)$, 
where $\omega(v,w) = \Im \la v, w \ra$. 
As the symplectic orthogonal space 
$V' := V^{\bot_\omega}$ coincides with 
$i V^{\bot_\R}$, where $\bot_\R$ denotes the orthogonal space 
with respect to the real scalar product 
$\Re \la v,w \ra$, the Lagrangian condition 
$V = V'$ is equivalent to $V = i V^{\bot_\R}$, 
which is equivalent to $V \oplus iV = \cH$ being an orthogonal direct 
sum. 
\item pairs $(\Delta, J)$ of modular objects  with $\Delta = \1$. 
\end{itemize}

(b) The canonical embedding of the symmetric space 
$\Conj(\cH)$: Note that 
\[ \zeta \: \Conj(\cH) \to \Hom_{\rm gr}(\R^\times, \AU(\cH)), 
\quad 
\zeta(J)(-1) := J, \quad \zeta(J)(e^t) =  \1 \quad \mbox{ for } \quad 
t \in \R \] 
defines a morphism of reflection spaces whose range is the 
set $\Hom_{\rm gr}(\R^\times, \AU(\cH))^\theta$ of $\theta$-fixed points. 
We likewise obtain morphisms of reflection spaces 
\[ \Conj(\cH) \to \Mod(\cH), \quad J \mapsto (\1, J), 
\qquad 
\Conj(\cH) \to \Stand(\cH),  J \mapsto \Fix(J) = \cH^J.\] 
\end{remark}

\begin{remark} For $\cH = \C^n$, the space 
$\Stand(\C^n) \cong \GL_n(\C)/\GL_n(\R)$ carries a natural 
symmetric space structure corresponding to $\tau(g) = \oline g$ on $\GL_n(\C)$ and 
given by 
\[ g\GL_n(\R) \sharp h \GL_n(\R) := g \tau(g^{-1}h) \GL_n(\R),
\quad \mbox{ resp., }\quad 
 g \R^n \sharp h \R^n = g \oline{g}^{-1} \oline h \R^n.\] 
 
This reflection structure is different from the one defined above 
by $s_{V_1} V_2 = V_1 \bullet V_2 = J_1 J_2 V_2$. 
In fact, if $J_1= J_2$, then $J_1 J_2 V_2 = V_2$ shows that $V_2$ is a fixed point 
of $s_{V_1}$ and thus $V_1$ is not isolated in $\Fix(s_{V_1})$, whereas 
this is the case for the reflections defined by $\sharp$ 
(cf.~Example~\ref{ex:1.3}(b)). 
\end{remark}

\subsection{Loos normal form of $\Stand(\cH)$} 

The unitary group $\U(\cH)$ acts on the reflection space 
$(\Stand(\cH), \bullet)$ by automorphisms 
and the morphism of reflection spaces 
\[ q \:  \Stand(\cH) \to \Conj(\cH), \quad q(V) := J_V \]  
is equivariant with respect to the conjugation action on 
$\Conj(\cH)$ which is transitive. 

For the involutive automorphism $\tau(g) := JgJ$ 
of $G := \U(\cH)$ we have $G^\tau \cong \OO(\cH^J)$. 
For $g_1, g_2 \in G$ and $V_1, V_2 \in \Stand_J(\cH)$, we have
\[ g_1 V_1 \bullet g_2 V_2 
= J_{g_1 V_1} J_{g_2 V_2} g_2 V_2 
= (g_1 J g_1^{-1})(g_2 J g_2^{-1}) g_2 V_2 
= g_1 \tau(g_1^{-1}g_2) V_2.\] 

\begin{proposition} Let $J \in \Conj(\cH)$ and 
write 
\[ \Stand_J(\cH) := \{ V \in \Stand(\cH) \: J_V = J\} = q^{-1}(J)\] 
for the $q$-fiber of $J$. Then the map
\[  \cP \: \U(\cH) \times \Stand_J(\cH) \to \Stand(\cH), \quad 
(g,V) \mapsto gV \] 
is surjective and factors to a bijection 
\[  \oline\cP \: \U(\cH) \times_{\OO(\cH^J)} \Stand_J(\cH) \to \Stand(\cH), \quad 
[g,V] \mapsto gV. \] 
It is an $\U(\cH)$-equivariant isomorphism of reflection spaces 
if the space on the left carries 
the reflection space structure given by 
\[ [g_1, V_1] \bullet [g_2, V_2] = [g_1 \tau(g_1^{-1}g_2), V_2].\] 
In particular, $\Stand_J(\cH)$ is a trivial reflection 
subspace of $\Stand(\cH)$ on which the product is given by 
$V_1 \bullet V_2 = V_2.$ 
\end{proposition}

We thus obtain a ``normal form'' of the reflection space 
$\Stand(\cH)$ similar to the one in Example~\ref{ex:1.3}(f). 

\begin{proof} The surjectivity of $\cP$ follows from  the transitivity of 
the action of $\U(\cH)$ on $\Conj(\cH)$, which in turn follows from 
the existence of an orthonormal basis of $\cH$ fixed pointwise by~$J$.
The second assertion 
follows from the fact that $\OO(\cH^J)$ is the stabilizer of $J\in \Conj(\cH)$. 
\end{proof}

The subset $\Stand_J(\cH)$ corresponds to the set of all 
positive selfadjoint operators $\Delta$ with 
$J\Delta J = \Delta^{-1}$. For $A := i\log \Delta$, this means that 
$JAJ = A$, so that $A$ corresponds to a skew-adjoint operator on 
the real Hilbert space $\cH^J$, hence, by the real version of Stone's Theorem to 
the infinitesimal generator of a continuous one-parameter group in $\OO(\cH^J)$. 
We thus obtain a bijection $\Stand_J(\cH) \to \Hom(\R, \OO(\cH^J))$.

\begin{remark} The dilations on $\Stand(\cH)$ corresponding to $V$ 
are implemented by $J_V$ and the unitary operators 
$(\Delta_V^{it})_{t \in \R}$. Since $\AU(\cH)$ 
acts on $\Conj(\cH)$, these one-parameter groups 
act naturally on $\Conj(\cH)$ but they do not give rise 
to a dilation space structure because they do commute with 
the stabilizer $\OO(\cH^J) \cong \U(\cH)_J$  of $J$. 
In addition, the dilation groups depend on the pair 
$(\Delta, J)$ and not only on $J$. 
\end{remark}

\subsection{Geodesics in $\Stand(\cH)$} 

In a reflection space, we have a canonical 
notion of a geodesics. Although we do not specify a topology 
on $\Stand(\cH)$, the space $\Conj(\cH)\subeq \AU(\cH)$ carries 
the strong operator topology, and this immediately provides a 
natural continuity requirement for geodesics in $\Stand(\cH)$.

\begin{proposition} \mlabel{prop:2.9}
{\rm({Geodesics in $\Stand(\cH)$})} 
Let $\gamma \: \R \to \Stand(\cH)$ be a geodesic 
with $\gamma(0) = V$ such that the corresponding geodesic 
$(J_{\gamma(t)})_{t \in \R}$ in $\Conj(\cH)$ is strongly continuous. 
Then there exists a strongly continuous unitary one-parameter group 
$(U_t)_{t \in \R}$ satisfying 
\begin{equation}
  \label{eq:invert}
J_V U_t J_V = U_{-t} \quad \mbox{ for } \quad t \in \R 
\end{equation}
such that 
\[ \gamma(t) = U_{t/2} V \quad \mbox{ for } \quad t \in \R.\] 
The $U_t$ are uniquely determined by the relation 
\begin{equation}
  \label{eq:u-rel}
J_{\gamma(t)} = U_{t/2} J U_{t/2}^{-1} = U_t J, \quad \mbox{ resp.} \quad 
U_t = J_{\gamma(t)} J.
\end{equation}
\end{proposition}

\begin{proof} Consider the base point $e := V \in \Stand(\cH)$ 
and the morphism 
\[ q \: \Stand(\cH) \to \Conj(\cH), \quad W \mapsto J_W \] 
of reflection spaces. By assumption, 
$\oline\gamma :=q \circ \gamma \: \R \to \Conj(\cH)$ 
is a strongly continuous geodesic, hence of the form  
$\oline\gamma(t) = U_t J_V$ 
for some continuous unitary one-parameter group $(U_t)_{t \in \R}$ 
(Theorem~\ref{thm:oeh}). Since the range of $\oline\gamma$ consists of 
involutions in $\AU(\cH)$, \eqref{eq:invert} follows. 
Further, 
\[ \gamma(t) = \gamma(t/2)\bullet \gamma(0) =  
 J_{\gamma(t/2)} J_V V = \oline\gamma(t/2) J_V V 
= U_{t/2} V.\] 
The relation \eqref{eq:u-rel} follows immediately from \eqref{eq:invert}.
\end{proof}

\begin{definition} We call a geodesic $\gamma \: \R \to \Stand(\cH)$ with 
$\gamma(0) = V$ {\it dilation invariant} if it is invariant under the corresponding 
one-parameter group $(\Delta_V^{it})_{t \in \R}$ of modular 
automorphisms, i.e., there exists an $\alpha \in \R$, such that 
\[ \Delta_V^{is} \gamma(t) = \gamma(e^{\alpha s} t) \quad \mbox{ for } \quad 
s,t \in \R.\] 
\end{definition}

\begin{proposition} \mlabel{prop:2.12} 
Consider the non-constant 
geodesic $\gamma(t) = U_{t/2}V$ through~$V$ and 
the unitary one-parameter group 
$W_s := \Delta_V^{-is/2\pi}$ implementing the dilations in~$V$. 
If $\gamma$ is dilation invariant, then there exists an 
$\alpha \in \R$ with 
\[ W_s U_t W_{-s} = U_{e^{\alpha s}t} \quad \mbox{ for } \quad 
t,s \in \R. \] 
Then 
\begin{equation}
  \label{eq:ax+b}
 U_{(0,-1)} := J_V \quad \mbox{ and }  \quad 
U_{(b,e^s)} := U_b W_{s} 
\end{equation}
defines an antiunitary representation
 of the group $G_\alpha := \R \rtimes_\zeta \R^\times$ with 
$\zeta(-1)x = -x$ and $\zeta(e^t)x = e^{\alpha t}x$. 
Conversely, for every antiunitary representation $(U,\cH)$ of 
$G_\alpha$, the restriction to 
$\{0\} \times \R^\times$ specifies a standard subspace $V$ 
and $\gamma(t) := U_{(t/2,1)} V$ 
is a dilation invariant geodesic with $\gamma(0) = V$. 
\end{proposition}

Note that, for $\alpha \not=0$, the group $G_\alpha$ is isomorphic 
to $\Aff(\R)$ and otherwise to 
$\R^2 \rtimes_\sigma \{\pm 1\}$ with 
$\sigma(-1)(x,y) = (-x,y)$. 

\begin{proof} The dilation invariance of $\gamma$ 
implies the existence of an $\alpha \in \R$ with 
$W_s \gamma(t) = \gamma(e^{\alpha s}t)$ for $s,t \in \R$. 
For the corresponding unitary one-parameter group $U$, this leads to 
\[ W_s U_t W_{-s} 
= W_s J_{\gamma(t)} J_V W_{-s} 
= J_{\gamma(e^{\alpha s}t)} J_V = U_{e^{\alpha s}t} \quad \mbox{ for } \quad s,t \in \R.\] 
Therefore 
\eqref{eq:ax+b} defines an antiunitary representation of $G_\alpha$. 
The converse is clear.
\end{proof}

\section{The order on $\Stand(\cH)$} 
\mlabel{sec:3}

As a set of subsets of $\cH$, the space $\Stand(\cH)$ 
carries a natural order structure, defined by set inclusion. 
We shall see below that non-trivial inclusions 
$V_1 \subset V_2$ arise only if both modular operators 
$\Delta_{V_1}$ and $\Delta_{V_2}$ are unbounded. Therefore 
inclusions of standard subspaces appear only if 
$\cH$ is infinite dimensional. Here a natural question 
is to understand when a non-constant geodesic $\gamma \: \R \to \Stand(\cH)$ 
is monotone with respect to the natural order on $\R$. 
In general this seems to be hard to characterize, but for 
dilation invariant geodesics, Proposition~\ref{prop:2.12}  
can be combined with the Borchers--Wiesbrock Theorem 
(\cite[Thms.~3.13, 3.15]{NO17}) which provides a complete answer 
in the case in terms of the positive/negative 
energy condition on the corresponding antiunitary representation of~$\Aff(\R)$. 

\subsection{Monotone dilation invariant geodesics}

\begin{lemma} \mlabel{lem:3.2}
If $V_1 \subset V_2$ is a proper inclusion of standard subspaces, 
then both operators $\Delta_{V_j}$, $j =1,2$, are unbounded. 
\end{lemma}

\begin{proof} Suppose that $\Delta_{V_1}$ is bounded. 
Then $S_1 := J_{V_1} \Delta_{V_1}^{1/2}$ is also bounded and 
thus $\cH = \cD(S_1) = V_1 + i V_1$. As $V_1 \subeq V_2$ 
and $V_2 \cap i V_2 = \{0\}$, the inclusion cannot be strict. 

If $\Delta_{V_2} = \Delta_{V_2'}^{-1}$ is unbounded, 
this argument shows that the inclusion $V_2' \subeq V_1'$ 
cannot be strict, hence $V_2' = V_1'$ and thus 
$V_2 = V_2'' = V_1'' = V_1$. 
\end{proof}

From \cite[Prop.~3.10]{Lo08} we recall: 
\begin{lemma} \mlabel{lem:3.3} 
If $V_1 \subeq V_2$ for two standard subspaces 
and $V_1$ is invariant under the modular automorphisms 
$(\Delta_{V_2}^{it})_{t \in \R}$, then $V_1 = V_2$. 
\end{lemma}

\begin{theorem} \mlabel{thm:3.1}  
Let $\gamma(t) = U_{t/2}V$ 
be a non-constant dilation invariant geodesic with $U_t = e^{itH}$. 
Then $\gamma$ is decreasing if and only if 
$H \geq 0$ and $W_s = \Delta_V^{-is/2\pi}$ acts non-trivially on 
$\gamma(\R)$. 
\end{theorem}

\begin{proof} First we assume that $(W_s)_{s \in \R}$ acts non-trivially on $\gamma(\R)$, 
which means that $\alpha \not=0$ in Proposition~\ref{prop:2.12}, so that 
we obtain an antiunitary representation of $\Aff(\R)$. 
Now the assertion follows from \cite[Thm.~3.13]{NO17} 
or \cite[Thm.~3.17]{Lo08}.

If $\alpha = 0$, i.e., $(W_s)_{s \in \R}$ commutes with 
$(U_t)_{t \in \R}$, then each $\gamma(t)$ is invariant under $(W_s)_{s \in \R}$. 
Therefore $\gamma$ cannot be monotone by 
Lemma~\ref{lem:3.3}. 
\end{proof}

\begin{remark} Comparing Proposition~\ref{prop:2.12} with \cite[Thm.~3.22]{NO17}, 
it follows that two different standard subspaces $V_0, V_1 \in \Stand(\cH)$ 
lie on a dilation invariant geodesic of type $\alpha \not=0$ 
if and only if they have a +-modular intersection 
(see \cite[\S 3.5]{NO17} for details). 
\end{remark}

\begin{prob} Find a characterization of the monotone 
geodesics in $\Stand(\cH)$. 
By Lemma~\ref{lem:3.2} it is necessary that $\Delta_V$ is unbounded. 
For dilation invariant geodesics,  
Theorem~\ref{thm:3.1} provides a characterization 
in terms of the positive/negative spectrum condition on~$U$. 
In this case the representation theory of $\Aff(\R)$ 
even implies that, apart from the subspace of fixed 
points, the operator  $\Delta_V$ 
must be equivalent to the multiplication 
operator $(Mf)(x) = x f(x)$ on some space $L^2(\R^\times, \cK)$, 
where $\cK$ is a Hilbert space counting multiplicity 
(\cite[\S 2.4.1]{NO17}, \cite[Thm.~2,8]{Lo08}). 
Note that the fixed point space $\cH_0 := \ker(\Delta_V - \1)$ 
leads to an orthogonal decomposition $\cH = \cH_0 \oplus \cH_1$ and 
$V= V_0 \oplus V_1$ with $\Delta_V = \1 \oplus \Delta_{V_1}$ such that 
$\Delta_{V_1}$ has purely continuous spectrum.
\end{prob}

Since it seems quite difficult to address this problem directly, 
it is natural to consider subspaces of $\Stand(\cH)$ which 
are more accessible. Such subspaces can be obtained 
from an antiunitary representation $(U,\cH)$ of a graded 
Lie group $(G,\eps_G)$ and a fixed $\gamma \in \Hom_{\rm gr}(\R^\times, G)$ 
from the image $\cO_V := G_1.V$ of the conjugacy class 
$G_1.\gamma$ under the $G_1$-equivariant morphism 
$\cV_U \: \Hom_{\rm gr}(\R^\times,G) \to \Stand(\cH)$ 
of dilation spaces, where $V := \cV_U(\gamma)$ (Corollary~\ref{cor:bgl}). 
Then 
\[  S_V := \{ g \in G_1 \: U_gV \subeq V \} \] 
is a closed subsemigroup of $G_1$ with $G_{1,V} = S_V \cap S_V^{-1}$ 
and $S_V$ determines an order structure on $G_1/G_{1,V}$ by 
$g G_{1,V} \leq g' G_{1,V}$ if $g \in g'S_V$ (see \cite[\S 4]{HN93} for background material 
on semigroups and ordered homogeneous spaces) for which the inclusion 
$G_1/G_{1,V} \into \Stand(\cH)$ is an order embedding. 
Note that $\cO_V$ is a $G_1$-equivariant quotient of $G_1/G_{1,\gamma} 
\cong G_1.\gamma \subeq 
\Hom_{\rm gr}(\R^\times,G)$. In the following 
subsection we explain how these spaces and their order structure 
can be obtained quite explicitly 
for an important class of examples including the important case where
$\gamma$ is a Lorentz boost associated to a wedge domain in Minkowski space.

\subsection{Conformal groups of euclidean Jordan algebras} 

\begin{definition}
 A finite dimensional real vector space $E$ endowed 
with a symmetric bilinear map $E \times E \to E, 
(a,b) \mapsto a\cdot b$ is 
said to be a {\it Jordan algebra} if 
\[ x \cdot (x^2 \cdot y) = x^2 \cdot ( x \cdot y) \quad \mbox{ for } \quad 
x,y \in E.\] 
If $L(x)y = xy$ denotes the left multiplication, then 
$E$ is called {\it euclidean} if the trace form 
$(x,y) \mapsto \tr(L(xy))$ is positive definite. 
\end{definition} 

Every  finite dimensional euclidean Jordan algebra 
is a direct sum of simple ones and simple euclidean Jordan algebras 
permit a nice classification (\cite[\S V.3]{FK94}). They are of the following types: 
\begin{itemize}
\item $\Herm_n(\K)$, $\K = \R,\C,\H$, with $A\cdot B := \frac{AB + BA}{2}$. 
\item $\Herm_3(\bO)$, where $\bO$ denotes the $8$-dimensional 
alternative division algebra of octonions. 
\item $\Lambda_n := \R \times \R^{n-1}$ with 
$(t,v)(t',v') = (tt' + \la v,v' \ra, tv' + t'v)$. Then 
the trace form is a Lorentz form, so that we can think of $\Lambda_n$ 
as the $n$-dimensional Minkowski space, where the first 
component corresponds to the time coordinate. 
\end{itemize}

Let $E$ be a euclidean Jordan algebra. 
Then $C_+ := \{ v^2 \: v \in E\}$ is a pointed closed convex cone in $E$ 
whose interior is denoted $E_+$. The Jordan inversion $j_E(x) = x^{-1}$ 
acts by a rational map on $E$. The {\it causal group} 
$G_1 := \Cau(E)$ is the group of birational maps on $E$ generated by the 
linear automorphism group $\Aut(E_+)$ of the open cone $E_+$, 
the map $-j_E$ and the group of translations. It is an index two 
subgroup of the {\it conformal group} 
$G := \Conf(E)$ generated by 
the structure group $H := \Aut(E_+)\cup - \Aut(E_+)$ of $E$ 
(\cite[Prop.~VIII.2.8]{FK94}), $j_E$ and the translations. 
For any $g \in G$ and $x \in E$ in which $g(x)$ is defined, 
the differential $\dd g(x)$ is contained in $H$. This specifies a 
group grading $\eps_G \: G \to \{\pm 1\}$ for which 
$\ker \eps_G = G_1= \Cau(E)$ (see \cite[Thm.~2.3.1]{Be96} and \cite{Be00} 
for more details on causal groups). 
It also follows that an element of $G$ defines a linear map 
if and only if it belongs to the structure group $H$. 

The conformal completion $E_c$ of $E$ is a compact smooth manifold 
containing $E$ as an open dense submanifold on which $G$ acts 
transitively. By analytic extension, it 
can be identified with the Shilov boundary of the corresponding 
tube domain $E + i E_+ \subeq E_\C$ (\cite[Thm.~X.5.6]{FK94}, 
\cite[Thms.~2.3.1, 2.4.1]{Be96}).  
The Lie algebra $\g$ of $G$ has a natural $3$-grading 
\[ \g = \g_1 \oplus \g_0 \oplus \g_{-1}, \] 
where $\g_1 \cong E$ corresponds to the space of constant 
vector fields on $E$ (generating translations), 
$\g_0 = \fh$ is the Lie algebra of $H$ 
(the structure algebra of $E$) 
which corresponds to linear vector field, and $\g_{-1}$ corresponds to 
certain quadratic vector fields which are conjugate under the inversion
$j_E$ to constant ones (\cite[Prop.~X.5.9]{FK94}). 

We have a canonical homomorphism 
\begin{equation}
  \label{eq:gamma-scalar}
\gamma \:\R^\times \to H\subeq G, \quad r \mapsto 
r \id_E 
\end{equation} 
which is graded because $\gamma(-1) = -\id_E$ maps $E_+$ to $-E_+$. 
Note that $h = \gamma'(0) \in \g_0$ satisfies 
\begin{equation}
  \label{eq:liegrad}
 \g_j = \{ x \in \g \: [h,x] =j x\} \quad \mbox{ for } \quad j=-1,0,1
\end{equation}
and 
\begin{equation}
  \label{eq:liegrad2}
  \Ad_{\gamma(r)} x_j = r^j x_j \quad \mbox{ for } \quad r \in \R^\times, x_j \in \g_j.
\end{equation}
For the involution $\tau(g) := \gamma(-1) g \gamma(-1)$, resp., 
$\tau(g)(x) =- g(-x)$ for $x, g(-x) \in E \subeq E_c$, we then have 
$G_\gamma = H\subeq G^\tau$ and since by \eqref{eq:liegrad2}
 both groups have the same Lie algebra $\g_0$, 
the homogeneous space $M := G/H$ is symmetric (Example~\ref{ex:1.3}(b)). 

To determine the homogeneous space $G_1/G_{1,\gamma} \cong G_1.\gamma = G.\gamma \in 
\Hom_{\rm gr}(\R^\times,G)$, we first determine the stabilizer group~$G_\gamma$ 
and derive some information on related subgroups. 

\begin{lemma} \mlabel{lem:hgam} 
The following assertions hold: 
\begin{itemize}
\item[\rm(i)] $H = G_\gamma = \{ g \in G \: \Ad_g h = h\}$ 
and $G_{1,\gamma} = H_1 = \Aut(E_+)$. 
\item[\rm(ii)] $-j_E$ is contained in the identity component $G_0$ of $G_1$. 
\item[\rm(iii)] $G^\tau = H \rtimes  \{\1,j_E\} 
= H_1 \rtimes \{\pm \1, \pm j_E\}$ and $G_1^\tau = H_1 \rtimes \{ \1, - j_E\}$. 
\end{itemize}
\end{lemma}

\begin{proof} (i) As $H \subeq G_\gamma\subeq G_h := \{ g \in G \: \Ad_g h = h\}$ 
is obvious, we 
have to show that every element $g \in G_h$ acts by a linear map 
on $E \subeq E_c$. Then the assertion follows from 
$\dd g(x) \in H$ for every $g \in G$ and $x \in E$. 

For every $v \in E \subeq E_c$ we have 
$\lim_{t \to \infty} \exp(-th).v= 0$, and this property determines the point $0$ 
as the unique attracting fixed point of the flow defined by 
$t \mapsto \exp(-th)$ on $E_c$. We conclude that $G_h$ fixes $0$. 
Likewise $\infty := j_E(0) \in E_c$ is the unique attracting fixed point 
of the flow defined by $t \mapsto \exp(th)$, and so 
$G_h$ fixes $\infty$ as well. This implies that $G_h$ acts on $E$ by 
affine maps fixing $0$, hence by linear maps (\cite[Thm.~2.1.4]{Be96}). 

(ii) Let $e \in E$ be the unit element of the Jordan algebra $E$. 
Then $-j_E(z) := - z^{-1}$ is the point reflection in the base point 
$ie$ of the hermitian symmetric space $T_{E_+} = E + i E_+$ 
with the holomorphic automorphism group 
$G_1 \cong \Aut(T_{E_+})$ (\cite[Thm.~X.5.6]{FK94}). 
Let $K := G_{1,ie}$ denote the stabilizer group of $ie$ in~$G_1$. Then 
$K$ is maximally compact in $G_1$ and its Lie algebra 
$\fk$ contains a central element $Z$ with $\exp(Z) = s$. This 
follows easily from the realization of $T_{E_+}$ as the unit 
ball $\cD \subeq E_\C$ of the spectral norm by the Cayley transform 
$p \: T_{E_+} \to \cD, p(z) := (z-ie)(z+ie)^{-1}$ which maps $ie$ to $0$ 
(\cite[p.~190]{FK94}). Now the connected circle group $\T$ acts on $\cD$ 
by scalar multiplications and the assertion follows. 

(iii) Since the Lie algebra of $G^\tau$ is $\g^\tau = \g_0$, the group 
$\Ad_{G^\tau}$ leaves $\fz(\g_0) = \R h$ invariant. 
If $g \in G^\tau$ and $\Ad_g h = \lambda h$, then 
$\Spec(\ad h) = \{ -1,0,1\}$ implies $\lambda \in \{ \pm 1\}$. 
If $\lambda = 1$, then $g \in G_h = H$, and if 
$\lambda = -1$, then we likewise obtain $g^{-1} j_E \in H$. 
\end{proof}

Let $\theta := \Ad_{-j_E} \in \Aut(\g)$ be the involution 
induced by the map $- j_E \in G_1 = \Cau(E)$ which is a Cartan 
involution of $\g$. It satisfies 
$\theta(h) = -h$ for the element 
$h = \gamma'(0)$ defining the grading 
and thus $\theta(\g_j) = \g_{-j}$ for $j =-1,0,1$. 
Therefore 
\[ C := C_+  + \theta(C_+)  \subeq \fq := \g_1 \oplus \g_{-1} \] 
is a proper $\Ad(H_1)$-invariant closed convex cone 
and $\Ad_h C = - C$ for $h \in H \setminus H_1$. 
By Lawson's Theorem (\cite[Thm.~7.34, Cor.~7.35]{HN93}, \cite{HO96}) 
$S_C := H_1 \exp(C)$ is a closed subsemigroup 
of $G_1$ which defines on $G_1/H_1$ a natural 
order structure by 
$g G_{1,V} \leq g' G_{1,V}$ if $g \in g'S_V$ 
which is invariant under the action of $G_1$ and 
reversed by elements $g \in G \setminus G_1$. 


\begin{theorem} \mlabel{thm:maxsem} {\rm(Koufany)} We have the equalities of semigroups 
\[ S_C = S_{E_+} := \{ g \in G_1 \:  g E_+ \subeq E_+ \} 
= \exp(C_+)  \Aut(E_+) \exp(\theta(C_+)) \] 
and in particular $S_C \cap S_C^{-1} = G_{E_+} = H_1 = \Aut(E_+)$. 
\end{theorem}

\begin{proof} By \cite[Thm.~4.9]{Ko95}, we have in the identity 
component $G_0$ of $G$ the equality 
\[ S_{E_+}\cap G_0  = \exp(C_+)  (\Aut(E_+)\cap G_0) \exp(\theta(C_+)) 
= (\Aut(E_+) \cap G_0)\exp(C) = S_C \cap G_0.\] 
The definition of $G_1 = \Cau(E)$ and Lemma~\ref{lem:hgam}(ii) imply that 
\begin{equation}
  \label{eq:g-g1}
G_1 = G_0 \Aut(E_+) \{\id, -j_E\} = G_0 \Aut(E_+). 
\end{equation}
This implies that 
\[ S_{E_+} = (S_{E_+}\cap G_0) \Aut(E_+) = \exp(C_+)  \Aut(E_+) \exp(\theta(C_+))\]  
and likewise, by Lemma~\ref{lem:hgam}(i),
\[ S_C = H_1 \exp(C) = \Aut(E_+)(\Aut(E_+) \cap G_0) \exp(C)
=  \Aut(E_+) (S_{E_+} \cap G_0) = S_{E_+}.\] 
\end{proof} 

In view of the fact that, in Quantum Field Theory 
standard subspaces are associated to domains in space-time, 
it is interesting to observe that the ordered space $(G_1/H_1, \leq)$ 
can be realized as a set of subsets of $E_c \cong G_1/P^-$, 
where $P^- = H_1  \exp(\g_{-1})$ is the stabilizer of 
$0 \in E\subeq E_c$ in~$G_1$ (\cite[Thm.~2.1.4(ii)]{Be96}). 

\begin{cor} \mlabel{cor:setreal} The map 
\[ \Xi \: G_1/H_1  \into 2^{E_c}, \quad 
g_1 H_1 \mapsto g_1 E_+ \] 
is an order embedding. 
\end{cor}

\begin{proof} 
Because of the $G_1$-equivariance of $\Xi$, this follows from 
$S_C = S_{E_+}$, which implies that $g_1 E_+ \subeq g_2 E_+$ is equivalent 
to $g_2^{-1} g_1 \in S_C$, i.e., to $g_1 H_1 \leq g_2 H_1$ in $G_1/H_1$.
\end{proof}

\begin{remark}
For the case where $E = \R^{1,d-1}$ is $d$-dimensional Minkowski 
space, the preceding results lead to the set $\cW = G.E_+$ 
of {\it conformal wedge domains} in the conformal completion $E_c$. 
It contains in particular the standard right wedge 
\[ W_R = \{ (x_0, x_1, \ldots, x_{d-1}) \: x_1 > |x_0| \} \subeq E \]
and all its images under the Poincar\'e group  
(cf.~\cite[Exs.~5.15]{NO17}). 

For $E = \R$, we obtain in particular 
the set of open intervals in $E_c \cong \bS^1$. 
\end{remark}

Finally, we connect the ordered symmetric space $G_1/H_1$ 
to $\Stand(\cH)$ by using antiunitary positive 
energy representations. 

\begin{definition} We call an antiunitary representation 
$(U,\cH)$ of $(G, \eps_G)$ a {\it positive energy representation} 
if there exists a non-zero $x \in C_+ \subeq \g_1$ for which 
the selfadjoint operator $-i\dd U(x)$ has non-negative spectrum.
\end{definition}

\begin{remark} \mlabel{rem:posen}
(a) For every antiunitary representation of $G$, the set 
\[ W_U := \{x \in \g \: -i \dd U(x) \geq 0 \} \] 
is a closed convex invariant cone in $\g$ which is invariant 
under the adjoint action of $G_1$ and any 
$g \in G\setminus G_1$ satisfies $\Ad_g W_U = - W_U$. 

(b) Since the Lie algebra $\g$ is simple, it contains a pair 
$W_{\rm min} \subeq W_{\rm max}$ of non-zero closed convex invariant cones 
and any other proper invariant convex cone $W$ satisfies 
\[ W_{\rm min} \subeq W \subeq W_{\rm max} \quad \mbox{ or } \quad 
W_{\rm min} \subeq -W \subeq W_{\rm max} \] (\cite[Thm.~7.25]{HN93}). 
As $W_{\rm min} \cap \g_1 = W_{\rm max} \cap \g_1 \in \{\pm C_+\}$ 
by \cite[Prop.~II.7, Thm.~II.10, Prop.~III.7]{HNO94}, for an antiunitary representation 
$(U,\cH)$ of $G$ with ${W_U \not=\{0\}}$ either 
$U$ or its dual $U^*$ satisfies the positive energy condition. 

For a concrete classification of 
antiunitary positive energy representation we refer to \cite{NO17b}.
By \cite[Thm.~2.11]{NO17}, this classification can be reduced to the 
unitary highest weight representations of $G_1$, resp., its identity component, 
which have been determined by Enright, Howe and Wallach. 
We refer to the monograph \cite{Ne00} for a systematic exposition 
of this theory. 
\end{remark}

\begin{theorem} \mlabel{thm:3.10b} 
Let $(U,\cH)$ be an antiunitary positive energy 
representation of $G$ for which $\dd U$ is non-zero, so that 
$W_U$ is a non-zero proper invariant cone. 
Let $V := \cV_U(\gamma) \in \Stand(\cH)$ be the standard subspace 
corresponding to $\gamma$ under the BGL map {\rm(Corollary~\ref{cor:bgl})}. 
Then 
\[  S_C = S_V := \{ g \in G_1 \: U_gV \subeq V \} \] 
and this implies that the BGL-map $\cV_U \: G_1.\gamma \cong 
G_1/H_1 \to \cO_V = U_{G_1}.V \subeq \Stand(\cH)$ 
defines an isomorphism of ordered dilation  spaces. 
\end{theorem}

\begin{proof}  By assumption $W_U$ is a proper closed convex invariant 
cone in $\g$ and in Remark~\ref{rem:posen} we have seen that 
$W_U \cap \g_1 \in \{ \pm C_+\}$, so that the positive energy 
condition leads to $W_U \cap \g_1 = C_+$. 
For $x \in C_+$ we have $[h,x] = x$ and 
$-i\dd U(x) \geq 0$, so that 
$\R x + \R h$ is a $2$-dimensional Lie algebra isomorphic to $\aff(\R)$. 
Therefore Theorem~\ref{thm:3.1} implies 
$\exp(\R_+ x) \subeq S_V$ and we even see that 
\begin{equation}
  \label{eq:c+}
\{ x \in \g_1 \: \exp(\R_+ x) \subeq S_V\} = W_U \cap \g_1 = C_+.
\end{equation}
As $\theta = \Ad_{-j_E}\in \Ad_{G_1}$ (Lemma~\ref{lem:hgam}), it leaves $W_U$ invariant. We conclude 
with Koufany's Theorem~\ref{thm:maxsem} that 
$S_{E_+} = S_C = \exp(C_+) H_1\exp(\theta(C_+)) \subeq S_V.$ 
Finally, we use the maximality of the subsemigroup 
$S_{E_+}\subeq G_1$ (Theorem~\ref{thm:3.10c}) to see that 
$S_C = S_V$. 
\end{proof}

\section{Open problems} 
\mlabel{sec:4}

\begin{prob} Let $(G,\eps_G)$ be a graded Lie group with 
two connected components, 
$\gamma \: \R^\times \to G$ be a graded smooth homomorphism 
and $(U,\cH)$ be an antiunitary representation of $G$. Then the 
$G_1$-invariant cone $W_U \subeq \g$ can be analyzed with the well-developed 
theory of invariant cones in Lie algebras (see~\cite[\S 7.2]{HN93} and also \cite{Ne00}). 
\begin{itemize}
\item 
Let $V := \cV_U(\gamma)$. 
Is it possible to determine when the order structure on the subset 
$U_{G_1}V= \cV_U(G_1.\gamma) \subeq \Stand(\cH)$ is non-trivial? 
Theorems~\ref{thm:3.1} and \ref{thm:3.10b} deal with very special cases. 
\item Is it possible to determine the corresponding order, which is given by 
the subsemigroup $S_V \subeq G_1$, intrinsically in terms of $\gamma$? 
Here the difficulty is that $G_1.\gamma \cong G_1/G_{1,\gamma}$ carries no 
obvious order structure. 
\end{itemize}
\end{prob}

\begin{prob} In several papers Wiesbrock develops 
a quite general program how to generate Quantum Field Theories, 
resp., von Neumann algebras of local observables 
from finitely many modular automorphism 
groups (\cite{Wi93, Wi93b, Wi97, Wi98}). 
This contains in particular criteria for three modular groups 
corresponding to three standard subspaces $(V_j)_{j =1,2,3}$ to generate 
groups isomorphic to the Poincar\'e group in dimension $2$ or to $\PSL_2(\R)$ 
(\cite[Thm.~3.19]{NO17}). On the level of von Neumann algebras 
there are also criteria for finitely many modular groups to 
define representations of $\SO_{1,3}(\R)^\uparrow$ or 
the connected Poincar\'e group $P(4)^\uparrow_+$ (\cite{KW01}). 

It would be interesting to see how these criteria can be expressed in 
terms of the geometry of finite dimensional totally geodesic dilation 
subspaces of $\Stand(\cH)$. 
\end{prob}

\appendix 

\section{Maximality of the compression semigroup of the cone} 

In this appendix we prove the maximality of 
the semigroup $S_{E_+}$ in the causal group $\Cau(E)$ 
of a simple euclidean Jordan algebra~$E$.

\begin{theorem} \mlabel{thm:3.10c} 
If $E$ is a simple euclidean Jordan algebra and $E_+ \subeq E$ the open 
positive cone, then the subsemigroup $S_{E_+}$ of $G_1 = \Cau(E)$ is maximal, 
i.e., any subsemigroup 
of $G_1$ properly containing $S_{E_+}$ coincides with $G_1$. 
\end{theorem}

\begin{proof} {\bf Step 1:} (Reduction to connected groups) 
First we recall from \eqref{eq:g-g1} that 
$G_1 = G_0 \Aut(E_+).$ 
As $\Aut(E_+) \subeq S_{E_+} \cap S_{E_+}^{-1}$, it therefore suffices to show that 
$S_{E_+}^0 := S_{E_+} \cap G_0$ is a maximal subsemigroup of 
the identity component $G_0$. \\

{\bf Step 2:} 
We want to derive the assertion from \cite[Thm.~V.4]{HN95}. 
In \cite{HN95} one considers a connected semisimple Lie 
group $G$, a parabolic subgroup $P$ and an involutive automorphism $\tau$ of $G$. 
In loc.~cit.\ it is assumed that the symmetric Lie algebra 
$(\g,\tau)$ is irreducible (there are no non-trivial $\tau$-invariant 
ideals) and that, for a $\tau$-invariant 
Cartan decomposition $\g = \fk \oplus \fp$ 
and the $\tau$-eigenspace decomposition $\g = \fh \oplus \fq$, 
the center of the Lie algebra 
\begin{equation}
  \label{eq:reg}
 \fh^a := (\fh \cap \fk) \oplus (\fq \cap \fp)
\quad \mbox{ satisfies}\quad 
\fz(\fh^a) \cap \fq \cap \fp \not=\{0\}.
\end{equation}
Then the conclusion of \cite[Thm.~V.4]{HN95} 
is that, if 
\begin{itemize}
\item $G^\tau P$ is open in $G$, 
\item the subsemigroup 
$S(G^\tau, P) := \{ g \in G \: gG^\tau P \subeq G^\tau P \}$ 
has non-empty interior, and 
\item $G = \la \exp_{G_\C} \g_\g \ra$ in the simply connected 
complex group $G_\C$ with Lie algebra~$\g_\C$, 
\end{itemize}
then $S(G^\tau, P)$ is maximal in $G$. 

We next explain how the assumption that $G_\C$ is simply 
connected can be weakened. 
Suppose that $G$ injects into its universal complexification 
$G_\C$ (which is always the case if it has a faithful finite dimensional 
representation). Let \break $q_\C \: \tilde G_{\C} \to G_\C$ denote the 
simply connected covering group and, 
as $G$ is connected, the integral subgroup 
$G^\sharp  := \la \exp_{\tilde G_\C} \g \ra \subeq \tilde G_\C$ 
satisfies $q_\C(G^\sharp) = G$ and $\ker q_\C$ is a finite central subgroup of 
$\tilde G_\C$. Consider the covering map $q := q_\C\res_{G^\sharp} \: G^\sharp \to G$. 
Then $P^\sharp := q^{-1}(P)$ 
is a parabolic subgroup of $G^\sharp$ and $G/P \cong G^\sharp/P^\sharp$. 
Let $\tau$ also denote the involution of $G^\sharp$ obtained by 
first extending $\tau$ from $G$ to a holomorphic involution of $G_\C$, 
then lifting it to $\tilde G_\C$ and then restricting to $G^\sharp$. 
Now $H' := q((G^\sharp)^{\tau}) \subeq G^\tau$ is an open subgroup 
satisfying $q((G^\sharp)^{\tau} P^\sharp) = H' P \subeq G^\tau P$. 
As $\ker(q) \subeq P^\sharp$ and $q$ is surjective, we even obtain 
\begin{equation}
  \label{eq:invrel}
(G^\sharp)^{\tau} P^\sharp = q^{-1}(H' P).
\end{equation}

By \cite[Thm.~V.4]{HN95}, 
$S_1 := S( (G^\sharp)^{\tau}, P^\sharp)$ is a maximal subsemigroup 
of $G^\sharp\subeq \tilde G_\C$. As $\ker q \subeq P^\sharp$, we have $S_1 = q^{-1}(S_2)$ 
for $S_2 := q(S_1)$. Further \eqref{eq:invrel} shows that 
\[ S_2 = \{ g \in G \: g H' P \subeq H' P \}.\] 
Now the maximality of $S_1$ in $G^\sharp$ 
immediately implies the maximality of $S_2$ in $G$. \\

{\bf Step 3:} (Application to causal groups of Jordan algebras) \\ 
First we verify the regularity condition~\eqref{eq:reg}. 
Here the Lie algebra $\g$ is simple, which implies 
in particular that $(\g,\tau)$ is irreducible. 

A natural Cartan involution of $\g$ is given by 
$\theta := \Ad_{-j_E}$ which satisfies $\theta(h) = -h$, hence 
$\theta(\g_j) = \g_{-j}$ for $j \in \{-1,0,1\}$. 
Then $\fh = \g^\tau = \g_0$ inherits the Cartan decomposition 
\[ \fh = \str(E) = \aut(E) \oplus L(E),\quad \mbox{ where } \quad 
L(x)y = xy \quad \mbox{ for } \quad x,y \in E,\] 
where $\aut(E)$ is the Lie algebra of the automorphism group 
$\Aut(E)$ of the Jordan algebra $E$, which coincides with 
the stabilizer group $H_e$ of the Jordan identity $e$ in~$H$. 
This shows that 
\[ \fh^a = \aut(E) \oplus \{ x-\theta(x) \: x \in \g_1 \}.\] 
Then the centralizer of $\aut(E)$ in $E\cong \g_1$ is $\R e$, and therefore 
\[ \{ x-\theta(x) \: x \in \g_1 \}^{\aut(E)} 
= \R (e - \theta(e)).\] 
To verify \eqref{eq:reg}, it remains to show that the element 
$e - \theta(e)$ is central in $\fh^a$, i.e., that it commutes with 
with all elements of the form $u - \theta(u)$, $u \in  \g_1$. 
As $\g_{\pm 1}$ are abelian subalgebras of $\g$, we have 
\[ [e-\theta(e), u - \theta(u)] 
= - [e,\theta(u)] - [\theta(e),u]
= -[e,\theta(u)] - \theta([e,\theta(u)]).\] 
So it suffices to show that $[e,\theta(u)] \in \fh^{-\theta}$. 
The vector field on $E$ corresponding to $\theta(u)$ is given by 
\[ X(z) = P(z,z)u \quad \mbox{ with }\quad 
P(x,y) = L(x)L(y)+ L(y)L(x)-L(xy)\] 
and therefore $[e,\theta(u)]$ corresponds to the linear vector field 
\[ E \to E, \quad z \mapsto \dd X(z)e = 2P(z,e)u = 2 L(z)u = 2 L(u) z.\] 
Since this is given by a Jordan multiplication, it belongs to 
$\fh^{-\theta}$ (\cite[Prop.~X.5.8]{FK94}). 
This proves that $(\g,\tau)$ satisfies the regularity 
condition~\eqref{eq:reg}. \\

{\bf Step 4:} (The maximality of $S_{E_+}$) 
Let $G_0$ denote the identity component of~$G$. 
Then the stabilizer $P := G_{0,e}$ of the Jordan identity 
$e \in E_+ \subeq E \subeq E_c$ is a parabolic 
subgroup and $G_0/P \cong E_c$ is a flag manifold of $G_0$. 

As above, let $\tau(g) := \gamma(-1)g\gamma(-1)$, resp., 
$\tau(g)(x) = -g(-x)$, as a birational map on~$E$,  
and observe that 
\begin{equation}
  \label{eq:h'}
(G_0)^\tau = (H \cap G_0)\{\1,-j_E\}
= (\Aut(E_+) \cap G_0)\{\1,-j_E\} 
\end{equation}
by Lemma~\ref{lem:hgam}. 

The subgroup $H' \subeq (G_0)^\tau$ from Step 2 consists 
of elements which are images of elements in the simply connected 
connected complex group $\tilde G_\C$ fixed under the involution $\tau$. 
The group $\tilde G_\C$ acts by birational maps on the complex Jordan algebra 
$E_\C$ and since $\tilde G_\C$ is simply connected, the subgroup 
$G_\C^\tau$ of $\tau$-fixed points in $G_\C$ is connected 
(\cite[Thm.~IV.3.4]{Lo69}). Therefore elements of $G_\C^\tau$ act on 
$E_\C$ by elements of the complex group $\exp_{G_\C}(\g_{0,\C})$, 
hence by linear maps. This shows that $H'$ acts on $E$ by linear maps 
and thus $H' \subeq \Aut(E_+)$ follows from \eqref{eq:h'}. 
Further, $H'$ contains $(G^\tau)_0 = \Aut(E_+)_0$ and thus 
$H'.e = E_+$, as a subset of $E_c$. This means that 
$H'P/P$ corresponds to $E_+$, and therefore the maximality 
of $S_{E_+} = S_2$ follows from Step~2. 
\end{proof}

\bibliographystyle{amsalpha} 

\end{document}